\documentclass[a4paper,10pt]{article}
\usepackage{amsmath}
\usepackage{amssymb}
\usepackage{amscd}
\usepackage{amsthm}
\usepackage{enumerate}
\usepackage{microtype}
\usepackage{geometry}
\newtheorem{theorem}{Theorem}
\newtheorem{lemma}[theorem]{Lemma}

\newtheorem{corollary}[theorem]{Corollary}

\newcommand{\cV}{\mathcal V}

\newcommand{\cS}{\mathcal S}
\newcommand{\cF}{\mathcal F}
\newcommand{\cG}{\mathcal G}

\newcommand{\fF}{\mathfrak F}

\newcommand{\PG}{{\mathrm{PG}}\,}

\newcommand{\PGL}{{\mathrm{PGL}}\,}

\newcommand{\Alt}{{\mathrm{Alt}}\,}

\newcommand{\FF}{{\mathbb F}}
\newcommand{\KK}{{\mathbb K}}
\newcommand{\Ui}{\overline{U_i}}
\newcommand{\Uj}{\overline{U_j}}
\newcommand{\Uz}{\overline{U_0}}
\newcommand{\Tr}{\mathrm{Tr}\,}
\newcommand{\Fix}{{\mathrm{Fix}\,}}

\newtheorem{proposition}[theorem]{Proposition}
\def\rif#1{(\ref{#1})}
\def\eqn#1$$#2$${\begin{equation}\label#1#2\end{equation}}

\begin{document}

\title{On some subvarieties of the Grassmann variety}

\author{L.Giuzzi, V.Pepe}

\maketitle
\begin{abstract}
Let $\cS$ be a Desarguesian $(t-1)$--spread of $\PG(rt-1,q)$, $\Pi$ a
$m$--dimensional subspace of $\PG(rt-1,q)$ and $\Lambda$ the linear
set consisting of the elements of $\cS$ with non--empty intersection
with $\Pi$. It is known that the Pl\"{u}cker embedding of the
elements of $\cS$ is a variety of $\PG(r^t-1,q)$, say $\cV_{rt}$. In
this paper, we describe the image under the Pl\"{u}cker embedding of
the elements of $\Lambda$ and we show that it is an $m$--dimensional
algebraic variety, projection of a Veronese variety of dimension $m$
and degree $t$, and it is a suitable linear section of $\cV_{rt}$.
\end{abstract}

\section{Introduction}
Let $V$ be a vector space
over a field $\FF$ and denote by $\PG(V,\FF)$ the usual projective geometry given by the lattice of subspaces of $V$.
If $\FF$ is the
finite field with $q$ elements $\FF_q$ and $\dim V=n$, then we
shall write, as customary, $\PG(n-1,q):=\PG(V,\FF_q)$.
Recall that if $\KK$ is a subfield of $\FF$ and  $[\FF:\KK]=t$
then $V$ is also endowed with
the structure of a vector space $\hat{V}$ of dimension $rt$ over $\KK$.
We shall denote by $\PG(V,\KK)$,
the projective geometry given by the lattice of the subspaces
of $V$ with $V$ is regarded as a vector space over $\KK$.

As each point of $\PG(V,\FF)$ corresponds to
a $(t-1)$--dimensional projective
subspace of $\PG(V,\KK)$, it is possible to represent the
projective space $\PG(V,\FF)$ as a subvariety $\cV_{rt}$
of the Grassmann manifold $\cG_{rt,t}$ of the
$t$--dimensional vector subspaces of $V$;
see \cite{V}.

A \emph{linear set} of $\PG(V,\FF)$ is a set of points  defined by an
additive subgroup of $V$. More in detail, let $\KK\leq\FF$, as above,
and suppose $W$ to be a vector space of dimension $m+1$ over
$\KK$. Then, the \emph{$\KK$--linear set} $\Lambda$
of $\PG(V,\FF)$ defined by $W$ consists
of all points of $\PG(V,\FF)$ of the form
\[ \Lambda=\{ \langle X\otimes\FF \rangle | X\in W \}. \]

Linear sets have been widely used to investigate
several different aspects of finite geometry, the two most remarkable
being blocking sets and finite semifields.
Following the approach pioneered by Schubert in
\cite{schubert}, it can be seen how the representation of subspaces
on the Grassmann manifold $\cG$ might provides an important tool for the study of their
behaviour and their intersections.

In the present paper, we are interested in the representation of a
$\KK$--linear set $\Lambda$ on $\cG$ and in determining the space of linear
equations defining it as linear section of $\cV_{rt}$.

Throughout this paper, when discussing Grassmannians we shall use
vector dimension for the spaces under consideration,
whereas we shall consider projective dimension
when discussing projective spaces.
As for algebraic varieties $V$ defined over a field $\FF$,
we shall always mean by \emph{dimension}
always mean the dimension of the variety
$\overline{\mathcal{V}}$, regarded over the algebraic closure
$\overline{\FF}$ of
$\FF$, defined by the same equations as $\mathcal V$.

\section{Grassmannians and Schubert varieties}

Fix an $n$ dimensional vector space $V=V_n(\FF)$ over $\FF$ and
write $G(n,k)$, $k<n$, for the set of all the $k$--subspaces of $V$.
It is well known that $G(n,k)$ is endowed with the structure
of a partial linear space and it can be embedded via the
Pl\"{u}cker map
\[ \varepsilon_{k}:\begin{cases}
  G(n,k) \rightarrow \bigwedge^k V \\
  W=\langle v_1,v_2,\ldots,v_k \rangle  \mapsto
  v_1 \wedge v_2 \wedge \cdots \wedge v_k
\end{cases}\]
in
the projective
space $\PG(\bigwedge^k V,\FF)$; here $\dim_{\FF}\bigwedge^k V = \binom{n}{k}$.
The image of $\varepsilon_{k}$, say $\cG_{nk}$, is an algebraic
variety of $\PG(\bigwedge^k V,\FF)$ whose
points correspond exactly to the totally decomposable
$1$--dimensional subspaces of $\bigwedge^k V$.

We now recall some basic properties of alternating multilinear
forms. Let $U$ be a vector
 space defined on $\FF$ and let $V^k:=\displaystyle\underbrace{V \times
   V \times \cdots \times V}_{k \text{ times}}$. A $k$--linear map $f:
 V^k \longrightarrow U$ is \emph{alternating} if $f(v_1,v_2,\ldots,v_k)=0$ when $v_i=v_j$ for some $i \neq j$.
 This implies that
 $\forall i,j \in \{1,2,\ldots,k\}$,
 $f(v_1,\ldots,v_i,\ldots,v_j,\ldots,v_k)=
 -f(v_1,\ldots,v_j,\ldots,v_i,\ldots,v_k)$.

\begin{theorem}[Universal property of the $k^{th}$ exterior power of a
  vector space, {\cite[Theorem 14.23]{ALA}}]
\label{ut}
 A map $f: V^k \longrightarrow U$ is alternating
  $k$--linear if, and only if, there is a
  linear map $\overline{f}:
  \bigwedge^k V \longrightarrow U$ with $\overline{f}(v_1 \wedge v_2
  \wedge \cdots \wedge v_k) = f(v_1,v_2,\ldots,v_k)$.
  The map $\overline{f}$ is uniquely determined.
\end{theorem}
\begin{corollary}
The $\FF$--vector space
\[ \Alt^k(V,U):=\{f: V^k \longrightarrow U| f \text{ is $k$--linear and alternating}\} \]
is isomorphic to the $\FF$--vector space
$\mathrm{Hom}(\bigwedge^kV,U)$.
\end{corollary}
In particular, let
$(\bigwedge^k V)'$ be the dual of $\bigwedge^k V$. Then,
$(\bigwedge^k V)'\simeq \Alt^k(V,\FF)$.
 Furthermore,
we also have
$(\bigwedge^k V)' \simeq \bigwedge^{n-k}V$.
Actually,  $(\bigwedge^k V)'$
is spanned by  linear maps of type acting on the pure vectors of
$\bigwedge^k V$ as
\[
v_1 \wedge v_2 \wedge \cdots \wedge v_k \mapsto v_1 \wedge v_2 \wedge \cdots \wedge v_k \wedge w_{k+1} \wedge \cdots \wedge w_n,
\]
and extended by linearity; see \cite[Chapter 5]{Gr} for more details.
 Here
$(w_{k+1},w_{k+2},\ldots,w_n)\in V^{n-k}$ is a fixed $(n-k)$--ple.

Let $F=A_1< A_2<\cdots< A_k$ be a proper flag consisting of $k$
subspaces of $V$.
The\-  \emph{Schubert variety}\- $\Omega(F)=\Omega(A_1,A_2,\ldots,A_k)$
induced by $F$ is the subvariety of $\cG_{nk}$ corresponding to
all $W\in G(n,k)$ such that $\dim W\cap A_i\geq i$ for all $i=1,\ldots,k$.
It is well known, see \cite[Corollary 5]{Kleiman_Laksov} and
\cite[Chapter XIV]{HP}, that a
Schubert variety is actually a linear section of the Grassmannian.
Furthermore, as the general linear group is flag--transitive,
all Schubert varieties defined by flags of the same kind,
i.e. with the same list of dimensions $a_i=\dim A_i$, turn out to be
projectively equivalent.

In the present work we shall be mostly concerned with Schubert varieties
of a very specific form, namely those for which $a_1=h\leq n-k$ and
$a_i=n-k+i$ for any $i=2,\ldots,k$.
Under these assumptions, $\Omega(A_1):=\Omega(A_1,A_2,\ldots,A_k)$
depends only on $A_1$ and corresponds to
the set of all $k$--subspaces with non--trivial intersection with
$A_1$.
Indeed, using once more \cite[\S 2, Corollary 5]{Kleiman_Laksov},
we see that $\Omega(A_1)$ is
the complete intersection of $\cG_{nk}$ with a linear subspace of
codimension $\binom{n-h}{k}$, meaning that the
subspace of the dual of $\bigwedge^k V$ of
the elements vanishing on $\Omega(A_1)$ has dimension $\binom{n-h}{k}$.

Using Theorem \ref{ut} we can provide
a description of the space of the linear maps vanishing on $\Omega(A_1)$.
For any $k$--linear map $f:V^k\to U$,
define the \emph{kernel} of $f$ as
\[ \ker f=\{w
\in V | f(w,v_2,\ldots,v_k)=0,  \forall v_i \in V \}. \]
It is straightforward to see that $\ker f$ is a subspace of $V$;
when $f$ is alternating and non--null,
 the dimension of $\ker f$ is trivially bounded from
above, as recalled by the following proposition.
\begin{proposition}
\label{p2}
The kernel of a non--null $k$--linear alternating
map $f$ of an $n$--dimensional vector space $V$ has dimension at most $n-k$.
\end{proposition}
\begin{proof}
By Theorem \ref{ut}, $f$ can be
regarded as a linear functional
$\overline{f}:\wedge^k V\to\FF$
where
\[ f(v_1,\ldots,v_k)=\overline{f}(v_1\wedge v_2 \ldots\wedge v_k). \]
Let $E=\langle v_1,\ldots,v_k\rangle$ and observe that
$f(E):=f(v_1,\ldots,v_k)=0$ when $\dim E<k$ or $\dim E\cap\ker f>0$
In particular, if $\dim\ker f>n-k$ we always have
$\dim E\cap\ker f>0$ for $\dim E \geq k$; this gives $f\equiv 0$.
\end{proof}

\begin{proposition}\label{kernel}
The subspace of $(\bigwedge^k V)'$
consisting of the linear forms  vanishing on
$\Omega(A_1)$ is isomorphic to the
subspace of the $k$--linear alternating maps whose kernel contains $A_1$.
In particular, if $h=\dim A_1\leq n-k$, then there exists a basis for this
subspace consisting
of maps whose kernel contains $A_1$ and has dimension $n-k$.
\end{proposition}
\begin{proof}
  Let $f: \bigwedge^k V \to \FF$ be a linear function vanishing
  on $\Omega(A_1)$. In particular, $f$ vanishes on all subspaces $E$ with
  $\dim E\cap A_1>0$. Thus, by the definition of kernel, $A_1\leq\ker f$.
  If $h>n-k$, then by Proposition \ref{p2}
  the only function vanishing on $\Omega(A_1)$
  is $f\equiv 0$ and there is nothing to prove.
  Let now $h \leq n-k$. By $(\bigwedge^k V)' \simeq \bigwedge^{n-k} V$, let us
   consider the linear maps:
   \[
v_1 \wedge v_2 \wedge \cdots \wedge v_k \mapsto v_1 \wedge v_2 \wedge \cdots \wedge v_k \wedge w_{k+1} \wedge \cdots \wedge w_n \]
   where $\{w_{k+1},w_{k+2},\ldots,w_n \}$ is a set of $n-k$ linearly
independent vectors such that $A_1 \leq \langle
w_{k+1},w_{k+2},\ldots,w_n \rangle$. The kernel of such a map is the
subspace $\langle w_{k+1},w_{k+2},\ldots,w_n \rangle$.
It is well known that the dimension
of the Pl\"{u}cker embedding of the $(n-k)$--subspaces containing a
fixed $h$--dimensional subspace is $\binom{n-h}{k}$.
As, by \cite[\S 2, Corollary 5]{Kleiman_Laksov}
this is also the dimension of the space of the linear functions vanishing on $\Omega(A_1)$, we have the
aforementioned linear maps can be used to also determine a basis for it.
\end{proof}

\section{Desarguesian spreads and linear sets}

 A $(t-1)$--\emph{spread} $\cS$ of $\PG(V,\FF)$ is a partition of the
 point-set of $\PG(V,\FF)$
 in subspaces of fixed projective dimension $t-1$. It is well known, see
 \cite{Segre}, that spreads exist if and only if $t|n$. Henceforth, let
 $n=rt$ and denote by $V_1$ a $\FF$--vector space such that
 $\dim_{\FF}V_1=n+1$ and $V < V_1$. Under these
 assumptions we can embed $\PG(V,\FF)$ as a
 hyperplane in $\PG(V_1,\FF)$. Consider
 the point--line geometry
 $A(\cS)$ whose points
 are the points of $\PG(V_1,\FF)$ not contained in $\PG(V,\FF)$ and whose
 lines of are the subspaces of $\PG(V_1,\FF)$
 intersecting $\PG(V,\FF)$ in exactly one spread element.
 We say that
 $\cS$ is a Desarguesian spread if $A(\cS)$ is a Desarguesian
 affine space.
 Here we shall focus on spaces defined over finite fields.
 We recall that, up to projective equivalence,
 Desarguesian spreads are unique and their
 automorphism group contains a copy of $\PGL(r,q^t)$.
 There are basically two main ways
 to represent a Desarguesian spread.

 Let $V:=V(r,q^t)$ be the standard $r$--dimensional vector space over
 $\FF_{q^t}$ and write $\PG(r-1,q^t)=\PG(V,q^t)$.
 When we regard $V$ as an $\FF_q$--vector space,
 $\dim_{\FF_q}V(r,q^t)=rt$; hence,
 $\PG(V,q)$ corresponds to  $\PG(rt-1,q)$;
 furthermore, a point $\langle (x_0,x_1,\ldots,x_{r-1}) \rangle$ of
 $\PG(r-1,q^t)$ corresponds to the $(t-1)$--dimensional subspace of
 $\PG(rt-1,q)$ given by $\{\lambda (x_0,x_1,\ldots,x_{r-1}), \lambda
\in \FF_{q^t}\}$. This is the so called the \emph{
  $\FF_q$--linear representation} of
$\langle (x_0,x_1,\ldots,x_{r-1}) \rangle$. The set $\cS$,
consisting of the $(t-1)$--dimensional subspaces of $\PG(rt-1,q)$ that are the linear representation of a point of $\PG(r-1,q^t)$, is
a partition of the point set of $\PG(rt-1,q)$ and it is the
$\FF_q$--linear representation of $\PG(r-1,q^t)$.

\begin{theorem}[\cite{BarlCof}]
The $\FF_q$--linear representation of $\PG(r-1,q^t)$ is a Desarguesian
spread of $\PG(rt-1,q)$ and conversely.
\end{theorem}

Throughout this paper we shall extensively
use the following result: if $\sigma$
is a $\FF_q$--linear collineation of $\PG(n-1,q^t)$ of order $t$, then the
subset $\Fix(\sigma)$ of all elements of $\PG(n-1,q^t)$ point--wise
fixed by $\sigma$ is a subgeometry isomorphic to $\PG(n-1,q)$.
This is a straightforward consequence of the fact that there is just one
conjugacy class of $\FF_q$--linear collineations of order $t$ in
$P\Gamma L(n,q^t)$, namely that of $\mu: X\to X^q$.
In particular,
all subgeometries $\PG(n-1,q)$ are projectively equivalent to the
set of fixed points of the map
 $(x_0,x_1,\ldots,x_{n-1}) \mapsto (x_0^q,x_1^q,\ldots,x_{n-1}^q)$.

\begin{lemma}\label{subspace}\cite[Lemma 1]{LunardonNS}
Let $\Sigma \simeq \PG(n-1,q)$ be a subgeometry of $\PG(n-1,q^t)$ and
let $\sigma$ be the $\FF_q$--linear collineation of order $t$ such
that $\Sigma=\Fix(\sigma)$. Then a subspace $\Pi$ of $\PG(n-1,q^t)$ is
fixed set--wise by $\sigma$ if and only if
$\Pi \cap \Sigma$ has the
same projective dimension as $\Pi$.
\end{lemma}

Take now $V$ to be a $rt$--dimensional projective space over
$\FF_{q^t}$ and let $U_i$ be the subspace of $V$ defined by the equations
$x_j=0$, $\forall j \notin \{ir+1,ir+2,\ldots,(i+1)r\}$.
Then, clearly, $V=U_0\oplus U_1 \oplus \cdots \oplus U_{t-1}$.

For any $(x_0,\ldots,x_{rt-1}) \in V$,
write $\mathbf{x}^{(i)}=x_{ir},\ldots, x_{(i+1)r-1}$, and consider
the $\FF_q$--linear collineation of $\PG(rt-1,q^t)$ of order $t$ given
by
\[ \sigma:
(\mathbf{x}^{(0)},\mathbf{x}^{(1)},\ldots,\mathbf{x}^{(t-1)}) \mapsto
(\mathbf{x}^{(t-1)q},\mathbf{x}^{(0)q},\ldots,\mathbf{x}^{(t-2)q}). \]
As seen above,
the set $\Fix\sigma$ is a subgeometry
$\PG(rt-1,q^t)$ isomorphic to $\PG(rt-1,q)$: in the remainder
of this section we shall denote such subgeometry just as
$\PG(rt-1,q)$.
In particular, we see that $\Fix\sigma=\PG(rt-1,q)$ consists of points of the form $\{(\mathbf{x},\mathbf{x}^{q},\ldots,\mathbf{x}^{q^{t-1}}),\mathbf{x}=x_0,x_1,\ldots,x_{r-1}; x_i \in \FF_{q^t}\}$.

Observe that we have
$\sigma(U_i)=U_{{i+1}\pmod t}$ and the semilinear collineation $\sigma$
acts cyclically on the $U_i$;
furthermore, for any $u\in U_0$, $u\neq0$,
we have $u^{\sigma^i}\in U_i$ and the set
$\{u^{\sigma^i} : i=1,\ldots, t\}$ is linearly independent.
In particular, the subspace
$\Pi^*_u=\langle
\mathbf{u},\mathbf{u}^{\sigma},\ldots,\mathbf{u}^{\sigma^{t-1}} \rangle$
has projective dimension $t-1$.
The set
$\cS^*=\{\Pi^*_u, u \in U_0\}$
consists of $(t-1)$--spaces and it
is a $\FF_q$--rational normal $t$--fold scroll of $\PG(rt-1,q^t)$ over
$\PG(r-1,q^t)=\PG(U_0,q^t)$.
Any subspace $\Pi^*_u$ is fixed set-wise by $\sigma$;
hence, by Lemma \ref{subspace}$, \Pi_u:=\Pi_u^* \cap \Sigma$ has the same
projective dimension $t-1$. The collection of $(t-1)$--subspaces
$\cS=\{\Pi_u | u \in U_0 \}$ is a spread of $\PG(rt-1,q)$, see \cite{Segre},
also called the \emph{Segre spread} of $\PG(rt-1,q)$.

\begin{theorem}[\cite{Desarguesian_spread}]
The Segre spread of $\PG(rt-1,q)$, obtained as the intersection with
$\PG(rt-1,q)$ with  a $\FF_q$--rational normal $t$--fold scroll
of $\PG(rt-1,q^t)$
over $\PG(r-1,q^t)$, is a Desarguesian spread.
\end{theorem}

The correspondence between linear representations and Segre
spreads is given as follows:
\[
\langle u \rangle_{\FF_q} \in \PG(U_0,q^t)\simeq\PG(r-1,q^t) \mapsto \langle
u,u^{\sigma},\ldots,u^{\sigma^{t-1}} \rangle \cap \PG(rt-1,q).
\]

Throughout this paper, we shall silently
identify the two aforementioned representations
of Desarguesian spreads.
In particular, a spread element will be regarded indifferently as a
$(t-1)$--subspace of $\PG(rt-1,q)$ of type
\[\{(\lambda u,\lambda^q
u^q,\ldots,\lambda^{q^{t-1}} u^{q^{t-1}}), \lambda \in \FF_{q^t}\} \]
and as its projection $\langle u \rangle \in \PG(U_0,q^t)$.

Fix now a Desarguesian $(t-1)$--spread $\cS$ of $\PG(rt-1,q)$ and fix
 also a subspace $\Pi$ of $\PG(rt-1,q)$ of projective dimension $m$.
The set $\Lambda$ of all elements of $\cS$ with non--empty intersection
with $\Pi$ is a \emph{linear set} of rank $m+1$. In other words,
$\Lambda$ may be regarded as the set of all points of
$\PG(r-1,q^t)$ whose coordinates are \emph{defined} by a vector space
$W$ over $\FF_q$ of dimension $m+1$.
Linear sets are used for several remarkable constructions in finite
geometry;
see \cite{linear_set} for a survey.

In order to avoid the trivial case $\Lambda=\cS$, we shall assume
$m+1\leq tr-t$. When $m+1=rt-t$ we shall
say that the linear set has \emph{maximum rank}.
Furthermore, as we are interested in \emph{proper} linear sets of
$\PG(r-1,q^t)$, that is linear sets which are not contained in any
hyperplane of $\PG(r-1,q^t)$, we have
$\langle  \Lambda \rangle = \PG(r-1,q^t) $; hence, $\Lambda$
must contain a frame of $\PG(r-1,q^t)$ and $m+1 \geq r$.
Throughout this paper
a linear set will always be understood
 to have rank $m+1$ with $r \leq m+1 \leq rt-t$.

We point out that,
when regarded point sets  of $\PG(r-1,q^t)$, linear sets provide a
generalization of the notion of subgeometry over $\FF_q$.
This is  shown by the following theorem.
\begin{theorem}[\cite{lunardon_polverino_04}\label{linearset}]
  Take
 $r\leq m+1 \leq t(r-1)$ and
 let $\Lambda$ be the projection in $\PG(m,q^t)$
 of\- a\- subgeometry $\Theta\cong \PG(m,q)$
 onto a $\PG(r-1,q^t)$. Then, $\Lambda$ is a $\FF_q$--linear set of
$\PG(r-1,q^t)$ of rank $m+1$.
Conversely, when $\Lambda$ is a linear set of
$\PG(r-1,q^t)$
of rank $m+1$, then either $\Lambda$ is a canonical subgeometry of
$\PG(r-1,q^t)$ or there exists a subspace $\Omega \cong \PG(m-r,q^t)$
\- of\- $\PG(m,q^t)$\- disjoint from $\PG(r-1,q^t)$ and a subgeometry
$\Theta\simeq\PG(m,q)$ disjoint from $\Omega$ such that $\Lambda$ is the
projection of $\Theta$ from $\Omega$ on $\PG(r-1,q^t)$.
\end{theorem}

In particular,
when $m+1=r$, we have $\Lambda \cong \PG(r-1,q)$ and this is
the unique linear set of rank $r$, up to
projective equivalence. When $m+1 > r$, there are several
non--equivalent linear set of any given rank; they
do not even have the same number of points. As $r$ and $t$
grow, the number of non--equivalent linear sets also grows,
so any attempt for classification  is
hopeless.

We end this section by showing that a linear set, when considered as a
subset of a Desarguesian spread, is a projection of a family of
maximal subspaces of a suitable Segre variety. We are aware that the
same result appears in the manuscript \cite{LVdV2014}, but we here present
a different and shorter proof which might be of independent interest.

The embedding:
$$
\PG(V_1,\FF)\times\PG(V_2,\FF)\times \cdots \times \PG(V_t,\FF)\rightarrow \PG(V_1\otimes V_2 \otimes \cdots \otimes V_t,\FF)\\
$$
$$(v_1,v_2,\ldots,v_t)\mapsto v_1\otimes v_2 \otimes \cdots \otimes v_t
$$
is the so called \emph{Segre embedding} of
$\PG(V_1,\FF)\times\PG(V_2,\FF)\times \cdots \times \PG(V_t,\FF)$ in
$\PG(V_1\otimes V_2\otimes\cdots\otimes V_t,\FF)$. Its image, comprising
the simple tensors
of $\PG(V_1\otimes V_2 \otimes \cdots \otimes V_t,\FF)$,
is an algebraic variety: the
\emph{Segre variety}.
 Suppose $t=2$ and $\dim V_i=n_i$ for $i=1,2$. Then,
 the Segre variety of $\PG(n_1n_2-1,\FF)$, say $\Sigma_{n_1n_2}$,
 contains two families of maximal subspaces:
 $\{\Pi_w, w \in V_1\}$, with $\Pi_w$ the $n_2$--dimensional vector space
 $\{w \otimes v, v\in V_2\}$, and $\{\Pi_u, u \in V_2\}$, with $\Pi_u$
 the $n_1$--dimensional vector space $\{v \otimes u, v\in V_1\}$. For an
 introduction to the study of this
 topic see, for instance, \cite[Chapter 25]{Thas}.

A $(t-1)$--regulus of rank $r-1$ of $\PG(rt-1,q)$ is a collection of
$(t-1)$--dimensional projective subspaces of type $\langle
P,P^{\gamma_1},\ldots,P^{\gamma_{t-1}} \rangle$, where $P \in \Gamma$,
$P^{\gamma_i} \in \Gamma_i$ with $\Gamma,\Gamma_1,\ldots,\Gamma_{t-1}$ being
$(r-1)$--dimensional subspaces spanning $\PG(rt-1,q)$ and the
collineations $\gamma_i$ defined such that $\gamma_i : \Gamma \rightarrow
\Gamma_i$, $i=1,2,\ldots,t-1$; see \cite{desarg}.
Let now $\Sigma_{rt} \subset \PG(rt-1,q)$ be
the Segre variety  of $\PG(r-1,q) \times \PG(t-1,q)$. We
recall the following result.

\begin{theorem}[\cite{desarg}]
\label{t7}
Any $(t-1)$--regulus of rank $r-1$ of a $\PG(rt-1,q)$ is the system of
maximal subspaces of dimension $t-1$ of the Segre variety $\Sigma_{rt}$ and
conversely.
\end{theorem}

Using theorems \ref{linearset} and \ref{t7}
we can formulate the following geometric description.
\begin{theorem}\label{segre}
  By field reduction, either the points of a $\FF_q$--linear set $\Lambda$ of
  $\PG(r-1,q^t)$ correspond to the system of
  maximal subspaces of dimension $t-1$ of the
  Segre variety $\Sigma_{rt}$
  or there exists a subspace $\Theta
  \cong \PG((m-r+1)t-1,q)$ of $\PG((m+1)t-1,q)$, disjoint from
  \-$\PG(rt-1,q)$ and a Segre variety $\Sigma_{m+1,t}$
  also disjoint from $\Theta$ such that the field reduction of the
  points of $\Lambda$
  corresponds to the projection of the $(t-1)$--maximal
  subspaces of $\Sigma_{m+1,t}$ from $\Theta$
  on $\PG(rt-1,q)$.
\end{theorem}
\begin{proof}
  Write $m+1=r$; then, $\Lambda\cong \PG(r-1,q)$. As all the Desarguesian
  subgeometries of the same dimension
  are projectively equivalent, we can suppose without loss of generality
  $\Lambda=\{\langle(x_0,x_1,\ldots,x_{r-1})\rangle_{\FF_{q^t}} | x_i \in \FF_q\}$.
  For any point $\mathbf{x}:=(x_0,x_1,\ldots,x_{r-1})$, the corresponding spread
  element in $\PG(rt-1,q)$
  is $\{(\lambda
  \mathbf{x},\lambda^q \mathbf{x},\ldots,\lambda^{q^{t-1}}
  \mathbf{x}), \lambda \in \FF_{q^t}\}$. Let
  $(1,\xi_1,\ldots,\xi_{t-1})$ be a basis for $\FF_{q^t}$ regarded
  as $\FF_q$-vector space and
  \[ \gamma_i: (\mathbf{x}^{(0)},\mathbf{x}^{(1)},\ldots,\mathbf{x}^{(t-1)})
  \mapsto
  (\xi_i \mathbf{x}^{(0)},\xi_i^{q}\mathbf{x}^{(1)},\ldots,\xi_i^{q^{t-1}}\mathbf{x}^{(t-1)}).
  \]
  Observe that
  the collineations $\gamma_i$ of $\PG(rt-1,q^t)$ all fix $\PG(rt-1,q)$
  set-wise; thus, for all $i=1,\ldots,t-1$ they act also as
  collineations of $\PG(rt-1,q)$.

  Let $P$ be the point $(\mathbf{x},\mathbf{x},\ldots,\mathbf{x})$, then 
 \[ \{(\lambda \mathbf{x},\lambda^q \mathbf{x},\ldots,\lambda^{q^{t-1}}
 \mathbf{x}), \lambda \in \FF_{q^t}\}=\langle P,P^{\gamma_1},\ldots,P^{\gamma_{t-1}}
 \rangle_{\FF_q}; \]
 so, by \cite{desarg}, the linear representation of
 a subgeometry is the system of
 maximal subspaces of dimension $t-1$ of the Segre variety $\Sigma_{rt}$

 If $m+1 > r$, then, by Theorem \ref{linearset} and by the well--known fact
 that the subspace spanned by any two elements of a Desarguesian
 spread is partitioned by spread elements, we have the statement.
\end{proof}

As a system of maximal subspaces of a Segre variety is always
a partition of the point-set of the variety, when we regard a linear set
$\Lambda$ of
$\PG(r-1,q^t)$ as a set of points of $\PG(rt-1,q)$, rather than  as a
particular collection of $(t-1)$--subspaces, we see
that $\Lambda$ is
either a Segre variety $\Sigma_{rt}$ or, for $m+1>r$ the projection of a Segre
variety $\Sigma_{m+1,t}$ on a $\PG(rt-1,q)$. We point out that
Segre varieties and their
projection share several combinatorial and geometric properties;
see, for example, \cite{Thas_VMald}.

\section{Representation of linear sets on the Grassmannian}

The image under the Pl\"{u}cker embedding of a
Desarguesian spread $\cS$ of $\PG(rt-1,q)$
determines the algebraic variety $\cV_{rt}$; this variety
actually lies in a subgeometry $\PG(r^t-1,q)$;
see
\cite{Segre,LunardonNS,V}.

 We briefly recall a few essential properties of
$\cV_{rt}$.
Let $V:=V(rt,q^t)$ and let $\varepsilon_t:G(rt,t)\to\PG(\bigwedge^t V,q^t)$ be
the usual Pl\"{u}cker embedding of
the $(t-1)$--projective subspaces of $\PG(rt-1,q^t)$
in $\PG(\bigwedge^t V,q^t)$. Denote by $\cG^*_{rt,t}$
the image of such embedding.
Recall that the subgeometry $\PG(rt-1,q)$ is the set of fixed points of
$ \sigma:
(\mathbf{x}^{(0)},\mathbf{x}^{(1)},\ldots,\mathbf{x}^{(t-1)}) \mapsto
(\mathbf{x}^{(t-1)q},\mathbf{x}^{(0)q},\ldots,\mathbf{x}^{(t-2)q})$.
As
$\PG(\binom{rt}{t}-1,q^t)=\PG(\bigwedge^t V,q^t)$ is spanned by its
totally decomposable vectors, that is its tensors of
rank $1$, we can define a collineation
$\sigma^*$ of $\PG(\bigwedge^t V,q^t)$ as
\[
\sigma^*: v_0 \wedge v_1 \wedge \cdots \wedge v_{t-1} \mapsto
v_0^{\sigma} \wedge v_1^{\sigma} \wedge \cdots \wedge v_{t-1}^{\sigma}.
\]
The collineation $\sigma^*$ turns out to be a $\FF_q$--linear collineation
of order $t$ of
$\PG(\binom{rt}{t}-1,q^t)$; hence,
the set of its fixed points is a subgeometry $\PG(\binom{rt}{t}-1,q)$.

By Lemma \ref{subspace}, a subspace of $\PG(rt-1,q^t)$ meets
$\PG(rt-1,q)$ in a subspace of the
same dimension if, and only if, it is fixed set-wise by $\sigma$. Clearly, any subspace of $\PG(rt-1,q)$ is contained
in exactly one subspace of $\PG(rt-1,q^t)$ of the same
dimension. Thus, the Grassmannian of the
$(t-1)$--subspaces of $\PG(rt-1,q)$, say  $\cG_{rt,t}$, can be
obtained as the intersection  $\cG_{rt,t}=\cG^*_{rt,t}(V)\cap\Fix(\sigma^*)$.

Recall now the decomposition
$V=U_0 \oplus U_1 \oplus \cdots \oplus U_{t-1}$ and let
$V^{\otimes t}:=\displaystyle \underbrace{V \otimes V \otimes \cdots
\otimes V}_{t\text{ times}}$. Denote by $I$ be the two-sided ideal of the
tensor algebra $\mathcal{T}(V)=\displaystyle\sum_{i=0}^{\infty}
V^{\otimes i}$ generated by $\{v \otimes v, v \in V\}$. As
$\bigwedge^t V=\frac{V^{\otimes t}}{V^{\otimes t} \cap I}$ and
$u_0\otimes u_1\otimes\cdots\otimes u_{t-1}\not\in I$ when
$u_i\in U_i$ and $u_i\neq 0$,
we can identify (with a slight abuse of notation) the element
$u_0\otimes u_1
\otimes \cdots \otimes u_{t-1}$ with
$u_0\wedge u_1 \wedge \cdots \wedge u_{t-1}$.
In particular,
we shall regard
$U_0 \otimes U_1 \otimes \cdots \otimes U_{t-1}$,
as a subspace of $\bigwedge^t V$,
write $\varepsilon_t(\Pi^*_{u})=u\otimes
u^{\sigma}\otimes \cdots \otimes u^{\sigma^{t-1}}$ and regard
$\cV_{rt}$ as a subvariety of $\cG_{rt,t}$.

Let now $\Sigma$
be the Segre variety of $\PG(r^t-1,q^t)$ consisting of the simple tensors
of $U_0\otimes U_1 \otimes \cdots \otimes U_{t-1}$,
and denote by  $\sigma^{\dagger}$ the $\FF_q$--linear collineation induced
by $\sigma$ on $\PG(U_0\otimes U_1 \otimes \cdots \otimes U_{t-1},q^t)$;
in particular,
$\sigma^{\dagger}(u_0\otimes u_1 \otimes\cdots \otimes u_{t-1})= u_{t-1}^q\otimes
u_0^{q} \otimes\cdots \otimes u_{t-2}^q $ and
$\cV_{rt}=\Sigma \cap\Fix(\sigma^{\dagger})$.
Actually, $\cV_{rt}$ is also as the image of the map
\[ \alpha:(x_0,\ldots,x_{r-1})\in\PG(r-1,q^t)\mapsto
(\prod_{i=0}^{t-1} x_{f(i)}^{q^i})_{f \in \fF}\in \PG(r^t-1,q) \subset \PG(r^t-1,q) \]
where $\fF=\{ f: \{0,\ldots, t-1\}\to\{0,\ldots,r-1\} \}$.
Here, $\alpha$ is the map that makes the following diagram commute:
$$\begin{array}{ccc}
                         & \alpha          &          \\
    \mathrm{PG}(r-1,q^t) & \longrightarrow &  \mathrm{PG}(r^t-1,q) \\
                           &               &                        \\
    \text{field reduction} \searrow             &                &             \nearrow   \varepsilon_t        \\
                      &       \mathrm{PG}(rt-1,q)                &                    \\
                      &          \cS=\text{ Desarguesian Spread}                      &            \\
  \end{array}
 $$

 Let now $\Sigma_{rt}$ be the Segre embedding
 of $\PG(r-1,\FF_q) \times \PG(t-1,\FF_q)$.
 It is well known that the Pl\"{u}cker embedding of
 a family of maximal subspaces of dimension $t-1$ of $\Sigma_{rt}$
 is a Veronese variety of dimension $r-1$ and degree $t$; see, for
 instance, \cite[Exercise 9.23]{H}.
 By Theorem \ref{segre}, the field reduction of a
 subgeometry $\PG(r-1,q)$ of $\PG(r-1,q^t)$ consists of the family of
 maximal subspaces of dimension $t-1$ of $\Sigma_{rt}$. Up to
 isomorphism, we can indeed assume $\PG(r-1,q)=\{(x_0,x_1,\ldots,x_{r-1}), x_i
 \in\FF_q\}$. The image under $\alpha$ of such a set is, clearly, a
 Veronese variety of dimension $r-1$ and degree $t$, the complete
 intersection of $\cV_{rt}$ with a subspace of dimension
 $\binom{r-1+t}{t}-1$.
 As a consequence of Theorem
 \ref{segre}, the image of a linear set of rank $m+1$ on
 $\cV_{rt}$ is the projection of a Veronese variety of dimension $m$
 and degree $t$. Hence, the dimension of such a variety is at most $m$.

 \begin{lemma}\label{min_m}
A minimal subspace $\Pi$ defining a linear set $\Lambda$ of
$\PG(rt-1,q)$ is spanned by points
$\{P_0,P_1,\ldots P_{m}\}$ such that $\forall i=0,1,\ldots,m$
the spread element containing $P_i$ intersects $\Pi$ only in $P_i$.
\end{lemma}
\begin{proof}
Let $\Pi$ be a minimal defining subspace for $\Lambda$ and
suppose that every spread element intersects $\Pi$ in at least a
line. Consider a hyperplane $\Pi'$ of $\Pi$.
As $\Pi'$ meets each spread element with non--empty intersection
with $\Pi$, we have that $\Pi'$ and $\Pi$
determine the same linear set and $\Pi'<\Pi$ --- a contradiction.
Thus, we can assume that $\Pi$ contains
at least a point $P$ such that the spread element through $P$
intersects $\Pi$ only in $P$. According to the terminology of
\cite{linear_set}, $P$ is a point of the linear set
of \emph{weight $1$}.
Suppose now that $\Pi$ is not spanned by its points of
weight $1$. Then, there is
a hyperplane $\Pi'$ in $\Pi$ containing
all of these points. A spread element either intersects $\Pi$ in only
one point $P$, hence $P \in \Pi'$, or it intersects $\Pi$ at least a
line;
thus it must intersect
also $\Pi'$. It follows $\Pi'$ and $\Pi$ determine the same linear set
and $\Pi'<\Pi$, contradicting the minimality of $\Pi$ again.
\end{proof}

From now on, when we say that a linear set $\Lambda$ has rank $m+1$, we suppose
that $m$ is the minimum possible; in particular
the defining subspace of $\Lambda$
 is taken to be of the type of Lemma \ref{min_m}.

\begin{proposition}
The image of a linear set $\Lambda$ of $\PG(rt-1,q)$ of rank $m+1$ on
the Grassmannian, hence on $\cV_{rt}$, is an algebraic variety of
dimension $m$, the projection of a Veronese variety of dimension $m$
and degree $t$.
\end{proposition}
\begin{proof}
By Theorem \ref{segre} and the above remarks,
the image of $\Lambda$, say $\cV$, is the
projection of a Veronese variety of dimension $m$.
Thus,  its dimension
is at most $m$. Let $\Pi=\langle P_0,P_1,\ldots,P_m\rangle$ be a
subspace determining $\Lambda$ and suppose that
each $P_i$ is of weight $1$.
Write $\Pi_i=\langle P_0,\ldots,P_i\rangle$ and let
$\Lambda_i$ be the linear set determined by $\Pi_i$, with
corresponding image  $\cV_i$. Then we
have $\cV_0 \subsetneq \cV_1 \subsetneq \cdots \subsetneq \cV_{m-1}
\subsetneq \cV$. Hence, the dimension of $\cV$ is $m$.
\end{proof}

It has been shown in \cite{V},
that the image of a linear set of a $\PG(1,q^t)$
is a linear section of $\cV_{2t}$. We can now generalize this result.
\begin{theorem}
\label{t5}
The image of a linear set $\Lambda$
of rank $m+1$ is the intersection of $\cV_{rt}$
with a linear subspace of codimension
at most $\binom{rt-m-1}{t}$.
In particular, this image is the intersection of the images of
$\binom{rt-m-1}{t}$ linear sets of maximum rank.
\end{theorem}
\begin{proof}
  Let $\Pi=\PG(W,q)$ be a defining subspace of $\PG(rt-1,q)$ for
  $\Lambda$.
  Write $\Omega=\Omega(W)$
  for the Schubert variety that is the Pl\"{u}cker
  embedding of the $t$-subspaces with non--trivial intersection with
  $W$. Then, the image of the linear set on $\cV_{rt}$ is $\Omega
  \cap \cV_{rt}$ and $\Omega$ is the complete intersection of the
  Grassmannian with a subspace of codimension $\binom{rt-m-1}{t}$.
  The statement now follows from Proposition \ref{kernel}.
\end{proof}

We now want to provide some insight on the space of all linear
equations vanishing on $\cV_{rt} \cap \Omega$.
Obviously, any subspace $\PG(m,q)$ of $\PG(rt-1,q) \subset \PG(rt-1,q^t)$
is determined by $n=rt-1-m$ independent $\FF_q$--linear equations. These can
always be
taken of the form
\begin{equation}
\label{eqns}
  \Tr(\displaystyle\sum_{i=0}^{r-1}a_{ji}x_i)=0,\qquad
  j=1,2,\ldots,n,
\end{equation}
where $\Tr:\FF_{q^t}\to \FF_q$ is
the usual trace function.

A spread element has non--empty
intersection with the $\PG(m,q)$ given by the equations in
\eqref{eqns} if, and only if, there exists a non--zero
$\lambda \in \FF_{q^t}$ such that
\[
  \Tr((\displaystyle\sum_{i=0}^{r-1}a_{ji}x_i)\lambda)=0 \qquad
  j=1,2,\ldots,n.
\]
In other words, this is the same as to require that
the $(rt-m-1)\times t$ matrix
$$
M=\begin{pmatrix}
  \displaystyle\sum_{i=0}^{r-1}a_{1i}x_i & (\displaystyle\sum_{i=0}^{r-1}a_{1i}x_i)^q & \cdots & (\displaystyle\sum_{i=0}^{r-1}a_{1i}x_i)^{q^{t-1}} \\
  \displaystyle\sum_{i=0}^{r-1}a_{2i}x_i & (\displaystyle\sum_{i=0}^{r-1}a_{2i}x_i)^q & \cdots & (\displaystyle\sum_{i=0}^{r-1}a_{2i}x_i)^{q^{t-1}} \\
  \cdots & \cdots & \cdots & \cdots \\
  \displaystyle\sum_{i=0}^{r-1}a_{ni}x_i & (\displaystyle\sum_{i=0}^{r-1}a_{ni}x_i)^q & \cdots & (\displaystyle\sum_{i=0}^{r-1}a_{ni}x_i)^{q^{t-1}}
\end{pmatrix}
$$
cannot have full rank; thus, each of its minors of order $t$
must be singular.
This condition corresponds to
a set of $\binom{rt-m-1}{t}$ equations,
each of them determining a hyperplane section of $\cV_{rt}$.
We remark that, as we expect from Proposition \ref{kernel},
every set of $t$ equations in \rif{eqns} determines a $(rt-t-1)$--dimensional
subspace containing $\PG(m,q)$, hence a linear set of maximum
rank containing the given one.

Clearly, not all of the equations obtained above are always
linearly independent of $\cV_{rt}$.
For instance, if there were a minor $M_0$ of order
$t-1$ in $M$ which is non--singular for any
choice of $x_i\neq 0$, then $rt-m-t$ equations would
suffice.

The rest of this paper is devoted to investigate
the dimension the space of the linear equations
vanishing on the image of a linear set on $\cV_{rt}$.
As we have already remarked, for any fixed rank $m+1 >r$,
there are many non--equivalent linear sets; here we propose
an unifying approach for linear sets of the same rank.

Let $\Pi=\PG(W,q)$ be a $m$--subspace defining a linear set of $\PG(rt-1,q)$, $\PG(W^*,q^t)$ be the $m$--dimensional projective
subspace of $\PG(rt-1,q^t)$ such that
$\PG(W^*,q^t) \cap \PG(rt-1,q)=\Pi$, and
$\Omega^*=\Omega(W^*) \subset \cG_{rt\,t}^*$ be
the Schubert variety of the $t$--subspaces with non--trivial
intersection with $W^*$.
Let also $\Sigma$ be the Segre variety of the simple tensors
of $U_0\otimes U_1 \otimes \cdots \otimes U_{t-1}$; recall that we
can identify $\Sigma$ with the set of
$\{u_0 \wedge u_1 \wedge \cdots \wedge u_{t-1}, u_i \in U_i\}$ in
$\bigwedge V^t$.
The lifting $\sigma^*$ of the $\FF_q$--linear collineation $\sigma$ to
$\PG(\binom{rt}{t}-1,q^t)$ acts as $\sigma^*(v_1 \wedge v_2 \wedge
\cdots \wedge v_t)=v_1^{\sigma} \wedge v_2^{\sigma} \wedge \cdots
\wedge v_t^{\sigma}$.
As $\sigma$ permutes the $U_i$'s,
$\sigma^*$ fixes $\Sigma$ set-wise.
Since $W^*$ is also fixed set-wise by $\sigma$, see Lemma \ref{subspace}, we
see
that $\Omega^*$ is set-wise fixed by $\sigma^*$.
Lemma \ref{subspace} guarantees
$\dim_{\FF_{q}}
\Omega \cap \cV_{rt}=\dim_{\FF_{q^t}} \Omega^* \cap \Sigma$, hence we shall
determine
$\dim_{\FF_{q^t}} \Omega^* \cap \Sigma$.

As there exists an embedding
$\phi:U\otimes U^{\sigma} \otimes \cdots \otimes U^{\sigma^{t-1}}\to\bigwedge^t V$,
there is also a canonical projection
$\phi':(\bigwedge^t V)'\rightarrow (U\otimes U^{\sigma} \otimes \cdots \otimes
U^{\sigma^{t-1}})'$, where $(\bigwedge^t V)'$ and
$(U\otimes U^{\sigma} \otimes \cdots \otimes
U^{\sigma^{t-1}})'$ are the duals of respectively  $\bigwedge^t V$ and
$U\otimes U^{\sigma} \otimes \cdots \otimes
U^{\sigma^{t-1}}$.
Let $\mathcal{F}$ be the subspace of $(\bigwedge^t V)'$ consisting
of the linear
functions vanishing on $\Omega^*$, and let $\phi'_/$ be the restriction of
$\phi'$ to $\mathcal{F}$. We are interested in the dimension of the
image of $\phi'_/$. The nucleus of $\phi_/'$ consists of the $t$--linear
alternating forms $f$ such that $\ker f$ contains $W^*$ and
$f(u_0,u_1,\ldots,u_{t-1})=0$ for all $u_i \in U_i$. Such a space is
isomorphic to the space of the $t$--linear forms $\overline{f}$
defined on a subspace $W^{\natural}$ complement of $W^*$ in $V$, with
$\overline{f}(\overline{u_1},\overline{u_2},\ldots,\overline{u_t})=0$
for all $\overline{u_i} \in \overline{U_i}$, where $\Ui$ is the
projection of $U_i$  on $W^{\natural}$ from $W^*$.

Observe that $\dim
\overline{U_i}=\dim \langle U_i,W^*\rangle\cap W^{\natural}=\dim\langle U_i,W^*\rangle+\dim W^{\natural}-\dim\langle U_i,W^*,W^{\natural}\rangle= \dim U_i-\dim (U_i\cap W^*)$, so $W^{*\sigma}=W^*$ and
$U_{i+1}=U_i^{\sigma}$ imply that $\dim \overline{U_i}=\dim
\overline{U_0}$, for all $i=1,\ldots,t-1$.

\begin{proposition}\label{projection2}
  We have $\dim U_i \cap W^*= h>0$ if, and only if,
  the linear set $\Lambda$ contains a
  $\FF_{q^t}$--projective subspace of dimension $h-1$. If the linear
  set is proper, that is it spans\- $\PG(r-1,q^t)$\- but\- it is not
  $\PG(r-1,q^t)$, this can occur only for $r\geq 3$.
  Furthermore, $h \leq \frac{m+1-r}{t-1}$  in general and
  $h = m+1-r$ if
  $t=2$.
\end{proposition}
\begin{proof}
  A proper linear set $\Lambda$, when considered as a subset of $\PG(r-1,q^t)$,
  spans the whole projective space; hence, the projection of $\Pi=PG(W,q)$ on
  $\PG(U_0,q^t)=\PG(r-1,q^t)$ necessarily
  spans $\PG(U_0,q^t)$. It follows
  that the projection of $\PG(W^*,q^t)$ also spans $\PG(U_0,q^t)$.
  For $t=2$,
  this implies that
  $\dim U_1 \cap  W^*=m+1-r$ and $m+1-r>0$ can occur only if $r \geq 3$,
  since $r \leq m+1 \leq t(r-1)$.

   Suppose now $t>2$ and
   let $Z=U_i\cap W^*$; then,
   $\langle Z^{\sigma^i}, i=0,\ldots,t-1\rangle \subseteq W^*$.
   For any $P\in\PG(Z,q^t)$,
   the projective $(t-1)$--space
   $\langle P,P^{\sigma},\ldots,P^{\sigma^{t-1}}\rangle \cap \PG(rt-1,q)$
   is a spread element completely contained in $\PG(W,q)$. In particular,
   $\PG(W,q)$
   contains a subspace of dimension $ht-1$ completely partitioned
   by spread elements. Thus there exists
   a projective subspace $\PG(h-1,q^t)$ completely contained in
   the linear set $\Lambda$. Write $m+1=ht+k$ and let $W^*_1$ be a subspace of
   dimension $k$ disjoint from $\langle
   Z^{\sigma^i},i=0,\ldots,t-1\rangle \subseteq W^*$. Then $\Lambda$ is a
   cone with vertex a $\PG(h-1,q^t)$ and base
   $\Lambda_1$, with $\Lambda_1$  the
   linear set induced by $W_1:=W^*_1 \cap \PG(rt-1,q)$.
   In order to have a proper linear set,
   we need $\dim \langle \Lambda_1 \rangle =r-h$ and $r-h >0$, so
   $k \geq r-h$; hence, $ht \leq m+1-r+h$. Since $m+1 \leq rt-t$,
we have $h \leq \frac{m+1-r}{t-1}$. We can have $h >0$
  only if $m+1 \geq t-1+r$, but we also have $m+1 \leq rt-t$, hence we get $rt-t \geq t-1+r$ and so $r \geq 3$.
\end{proof}

\begin{theorem}
Let $c:=\dim \overline{U_i}$. The map $\phi_/$  is injective
if, and only if, $m+1 > rt-t-c$. This is always the case for $t=2$,
$(r,t)=(2,3)$ and for $t \geq 3$ with $m+1 > tr-t-1-\frac{2}{t-2}$.
\end{theorem}
\begin{proof}
The kernel of $\phi_/$ is
the space of
the  alternating $t$--linear forms defined on the vector space $W^{\natural}$ of
dimension $rt-m-1$ and such that $f(u_0,u_1,\ldots,u_{t-1})=0$
$\forall u_i \in \overline{U_i}$ or, equivalently,
the space of of the linear forms defined on
$\bigwedge ^t W^{\natural}$  vanishing on
all the points that are the Pl\"{u}cker embedding of
a $t$--space with non-trivial
intersection with each $\overline{U_i}$.
For $t+c > rt-m-1$, every
$t$--subspace intersects every $\overline{U_i}$ non--trivially. This implies $f\equiv 0$ and $\phi$ is injective.
 By Proposition \ref{projection2}, $\frac{rt-m-1}{t-1}
\leq c \leq r$ and $c=2r-m-1$ for $t=2$. Hence, when $t=2$, the
condition $m+1 > rt-t-c=2r-2-2r+m+1$ is always fulfilled. Suppose
now $t \geq 3$. By Proposition \ref{projection2},
we have $rt-t-c \leq rt-t-\frac{rt-m-1}{t-1}$; hence,
$m+1 > rt-t-\frac{rt-m-1}{t-1}$ implies $m+1 > rt-t-c$. Thus,
$m+1 > rt-t-\frac{rt-m-1}{t-1}$ if, and only if,
$m+1 > rt-t-1-\frac{2}{t-2}$. When $t=3$, this is equivalent to $m+1 > 3r-6$,
a condition
which is obviously always fulfilled
for $r=2$.

 If $t+c \leq rt-m-1$, then the image via the Pl\"{u}cker embedding of the
$t$--spaces with non--trivial intersection with a $\overline{U_i}$ is a
Schubert variety cut on the Grassmannian by
a linear subspace of codimension
$\binom{rt-m-1-c}{t}$; hence, the dimension of the kernel of the map
$\phi_/$ is at least $\binom{rt-m-1-c}{t} \geq 1$.
\end{proof}

\begin{corollary}
Let $\PG(W,q) \subset \PG(rt-1,q)$ be the $m$--dimensional
subspace
defining a linear set $\Lambda$ and $\PG(W^*,q^t)$ be the unique subspace
of $\PG(rt-1,q^t)$ such that $\PG(W^*,q^t)\cap \PG(rt-1,q)=\PG(W,q)$.
Take $W^{\natural}$ such that $V(rt,q^t)=W^*\oplus W^{\natural}$ and
let also $\Ui$ be the
projection of $U_i$ on $W^{\natural}$. Write $c=\dim\Ui$.
Then, the image of $\Lambda$ is
the complete intersection of $\cV_{rt}$ with a linear subspace of
codimension $\binom{rt-m-1}{t}$ if, and only if, $m+1 > rt-t-c$. This is
always the case for $t=2$, $(r,t)=(2,3)$ and for $t \geq 3$ and $m+1 >
tr-t-1-\frac{2}{t-2}$. If $m+1 \leq rt-t-c$, then the image of
$\Lambda$ is the complete intersection of $\cV_{rt}$ with a linear
subspace of codimension $\dim \langle u_0 \wedge u_1 \wedge \ldots
\wedge u_{t-1}, u_i \in \overline{U_i}\rangle <\binom{rt-m-1}{t}$.
\end{corollary}
We can provide a complete description for the case $t=3$.
\begin{theorem}
\label{t=3}
Let $t=3$, $r>2$ and $m+1 \leq 3r-3-c$.
Then, the codimension of
$\langle u_0 \wedge u_1 \wedge u_{2}, u_i \in \overline{U_i}\rangle$
in $\bigwedge ^{t}W^{\natural}$ is
$3\binom{3r-m-1-c}{3}$.
\end{theorem}
\begin{proof}
As the projection of $\PG(W^*,q^t)$ on $\PG(U_0,q^t)$ spans
$\PG(U_0,q^t)$ we have $\dim \langle U_i,U_j \rangle \cap
W^*=m+1-r$; hence, $\dim \langle \Ui \Uj \rangle= 2r-m-1+r=3r-m-1=\dim
W^{\natural}$. Thus, $\langle \Ui \Uj \rangle=W^{\natural}$ for any $i\neq j$.
Let $\Omega_i$ be the Schubert variety of the $t$--subspaces with
non--trivial intersection with $\Ui$ and let $\cF_i$ the space
of the linear functions defined on $\bigwedge ^{t}W^{\natural}$ vanishing on
$\Omega_i$. By a slight abuse of notation,  identify the elements of
$\cF_i$ with the corresponding trilinear alternating maps
defined on $W^{\natural} \times W^{\natural} \times W^{\natural}$; the kernel of any element of
$\cF_i$ contains $\Ui$. Suppose $f_i+f_j=0$ with $f_i \in
\cF_i$, $f_j \in \cF_j$, $i \neq j$. Then, the kernel
of $f_i$ contains $\langle \Ui,\Uj \rangle=W^{\natural}$, so
$f_i=f_j=0$. Suppose now $f_0+f_1+f_2=0$, with $f_i \in \cF_i
\setminus\{0\}$ and $i=0,1,2$. For every $u_2 \in \overline{U_2}$,
$f_0(\cdot,\cdot,u_2)$ is a bilinear map vanishing on $\langle
\Uz,\overline{U_1} \rangle=W^{\natural}$; hence, it
is identically $0$ and the kernel of $f_0$
would  contain $\langle\Uz,\overline{U_2}\rangle=W^{\natural}$. This would
imply $f_0=0$, a contradiction. Hence $\dim \langle \cF_1,\cF_2,\cF_3\rangle=3\dim\cF_i$
\end{proof}

\begin{corollary}
Let $t=3$ and $r>2$. Suppose $\PG(W,q) \subset \PG(3r-1,q)$ to be the
$m$--subspace defining the linear set $\Lambda$. Let also
$\PG(W^*,q^3)$ be the
unique subspace of $\PG(3r-1,q^3)$ such that $\PG(W^*,q^3)\cap
\PG(3r-1,q)=\PG(W,q)$ and take $W^{\natural}$ such that $V(3r,q^3)=W^*\oplus W^{\natural}$.
Denote by
$\Ui$ the projection of $U_i$ on $W^{\natural}$ and write $c=\dim\Ui$.
Assume also $m+1 \leq 3r-3-c$.
Then, the image of $\Lambda$ is the complete intersection of
$\cV_{r,3}$ with a linear subspace of codimension
$\binom{3r-m-1}{3}-3\binom{3r-m-1-c}{3}$.
\end{corollary}

When $t>3$ and $m+1 \leq 3r-3-\dim\Ui$, it is not possible, in
general, to provide a formula for the codimension of the image of a
linear set on $\cV_{rt}$ depending only on $m$, as  shown by the
following example.

In $\PG(5,q^4)$, take the linear set $\Lambda_1$ of rank $9$ given by
$\{(x,x^q,y,y^q,y^{q^2},z), x,y \in \FF_{q^4}, z \in \FF_q\}$. The
subspace $W_1$ of $\PG(23,q)$ defining $\Lambda_1$ is
\[ \{(x,x^q,y,y^q,y^{q^2},z,x^q,x^{q^2},y^q,y^{q^2},y^{q^3},z,x^{q^2},x^{q^3},y^{q^2},y^{q^3},y,z,x^{q^3},x,y^{q^3},y,y^q,z),
x,y \in \FF_{q^4}, z \in \FF_q\}; \]
hence, the subspace $W_1^*$ of
rank $9$ of $\PG(23,q^4)$ containing $W_1$ is
\[ \{(x_1,x_2,x_5,x_6,x_7, x_9, x_2,x_3,x_6,x_7,x_8,x_9,x_3,x_4,x_7,x_8,x_5,x_9,x_4,x_1,x_8,x_5,x_6,x_9),
x_i \in \FF_{q^4}\}. \]
A complement is
\[ W_1^{\natural}=\{(0,0,0,0,0,0,y_1,0,y_2,y_3,0,y_4,y_5,0,y_6,y_7,y_8,y_9,y_{10},y_{11},y_{12},y_{13},y_{14},y_{15}), y_i \in \FF_{q^4}\}. \]
Let $\Ui$ be the projection of
$U_i$ on $W_1^{\natural}$. By a straightforward calculation, we get $c=\dim
\Ui=6$, $\dim \overline{U_0}\cap \overline{U_1}=\overline{U_0}\cap
\overline{U_3}=1$ and $\overline{U_0}\cap \overline{U_2}=0$. Then, the
number of equations defining the image of $\Lambda_1$ on $\cV_{6,4}$
is
$\binom{rt-m-1}{t}-4\binom{rt-m-1-c}{t}+4\binom{rt-m-1-2c+1}{t}=865$.

Consider now the following linear set $\Lambda_2$ of the same rank:
$\{(x,y,y^q,z,z^q,z^{q^2}),x\in \FF_{q^2}, y \in \FF_{q^4}|\Tr(y)=0,z
\in \FF_{q^4}\}$, where $\Tr:\FF_{q^4}\rightarrow \FF_q$
is the trace function.
In $\PG(23,q)$, we have
\[ \{(x,y,y^q,z,z^q,z^{q^2},x^q,y^q,y^{q^2},z^q,z^{q^2},z^{q^3},x,y^{q^2},-y-y^q-y^{q^2},z^{q^2},z^{q^3},z,x^q,-y-y^q-y^{q^2},y,z^{q^3},z,z^q)\}; \]
 hence in $\PG(23,q^4)$ we get
$W_2^*=\{(x_1,x_3,x_4,x_6,x_7,x_8,x_2,x_4,x_5,x_7,x-8,x_9,x_1,x_5,-x_3-x_4-x_5,x_8,
x_9,x_6,x_2,-x_3-x_4-x_5,x_3,x_9,x_6,x_7),x_i\in \FF_{q^4}\}$.
 A complement is
\[ W_2^{\natural}=\{(0,0,0,0,0,0,0,y_1,0,y_2,y_3,0,y_4,y_5,y_6,y_7,y_8,y_9,y_{10},y_{11},y_{12},y_{13},y_{14},y_{15}),
y_i \in \FF_{q^4}\}. \]
 We see that $c=\dim \Ui=6$, $\dim
\overline{U_0}\cap \overline{U_1}=\overline{U_0}\cap \overline{U_3}=0$
and $\overline{U_0}\cap \overline{U_2}=1$. Thus, the number of equation
defining the image of $\Lambda_2$ on $\cV_{6,4}$ is
$\binom{rt-m-1}{t}-4\binom{rt-m-1-c}{t}+2\binom{rt-m-1-2c+1}{t}=863\neq865$.

{\bf Remark}. Even if it is not possible to provide a formula for the
codimension of the image of a linear set on $\cV_{rt}$ depending only
on $m$ for $t>3$ and $m+1 \leq 3r-3-\dim\Ui$, the above arguments show
a possible way to actually determine its value on a case--by--case basis,
as this codimension is, in general, the same as $\dim \langle u_0
\wedge u_1 \wedge \ldots \wedge u_{t-1}, u_i \in
\overline{U_i}\rangle$.

\vspace{5pt}
\noindent
 Addresses of the authors:
\vskip.1cm 
\centerline{
\begin{tabular}{l@{\qquad}l}
\begin{minipage}[t]{7cm}
{\bfseries Valentina Pepe}\\[.1cm]
Dipartimento di Scienze di Base \\
\hbox to 1cm{ } e Applicate per l'Ingegneria\\[.1cm]
La Sapienza University of Rome\\[.1cm]
Via Scarpa 16\\
I-00161 Roma \emph{(Italy)}\\[.1cm]
{\tt valepepe@sbai.uniroma1.it}
\end{minipage} &
\begin{minipage}[t]{7cm}
{\bfseries Luca Giuzzi}\\[.1cm]
DICATAM \\
 Section of Mathematics \\[.1cm]
Universit\`a degli Studi di Brescia\\[.1cm]
Via Branze 43\\
I-25123 Brescia \emph{(Italy)}\\[.1cm]
{\tt luca.giuzzi@unibs.it}
\end{minipage}
\end{tabular}
}
\end{document}